\documentclass[review,3p,times]{elsarticle}
\usepackage{}
\usepackage{bbm}
\usepackage{amsfonts}
\usepackage{mathrsfs}
\usepackage{amsmath}
\usepackage{tikz}
\usepackage{newlfont}
\usepackage{float}
\usepackage{graphicx}
\usepackage[colorlinks,linkcolor=red,anchorcolor=blue,citecolor=green]{hyperref}
\usepackage{listings}
\usepackage{amsthm}
\usepackage{lineno}
\usepackage[linesnumbered, ruled, lined,boxed,commentsnumbered]{algorithm2e}[1]
\usepackage{amssymb}
\usepackage{subfigure}
\usepackage{epstopdf}
\usepackage{graphicx}
\usepackage[matrix,arrow]{xy}

\SetKwInput{KwIn}{Input}
\SetKwInput{KwOut}{Output}
\biboptions{sort&compress}
\theoremstyle{plain}
  \newtheorem{thm}{Theorem}[section]
  \newtheorem{lem}[thm]{Lemma}
  \newtheorem{prop}[thm]{Proposition}
  \newtheorem{cor}[thm]{Corollary}
\theoremstyle{definition}
  \newtheorem{defn}[thm]{Definition}
  
  \newtheorem{exmp}[thm]{Example}
  \newtheorem{rem}[thm]{Remark}

\theoremstyle{break}

\theoremstyle{definition}

  \usepackage{diagbox}
\makeatletter
\newcommand\figcaption{\def\@captype{figure}\caption}
\newcommand\tabcaption{\def\@captype{table}\caption}
\makeatother
\allowdisplaybreaks
\begin{document}
\begin{frontmatter}

\title{$S^{\ast}$-well-filtered spaces and $d^{\ast}$-spaces\tnoteref{label2}}

\author{Nana Han\fnref{label1}}
\ead{hnnmaths@163.com}
\author{Siheng Chen\fnref{label3}}
\ead{mathlife@sina.cn}
\author{Qingguo Li\fnref{label1}\corref{cor1}}
\ead{liqingguoli@aliyun.com}
\cortext[cor1]{Corresponding author}
\address[label1]{School of Mathematics, Hunan University, Changsha, Hunan, 410082, PR China}
\address[label3]{School of Mathematics and Statistics, Changsha University of Science and Technology, Changsha, Hunan, 410114, PR China}
\tnotetext[label2]{This work was supported by the National Natural Science Foundation of China (12231007), National Natural Science Foundation of China (12401594) and Scientific Research Foundation of Hunan Provincial Education Department (25C0114).}
\begin{abstract}
  Recently, Xu proposed a strongly well-filtered space in \cite{x2026} and systematically investigated some of its properties and characterizations. In this paper, we introduce a new class of $T_{0}$-spaces called $S^{\ast}$-well-filtered spaces, which is strictly larger than the class of strongly well-filtered spaces. First, we establish some connections among $S^{\ast}$-well-filtered spaces, $d^{\ast}$-spaces and weak well-filtered spaces. Then it is demonstrated that for any dcpo $P$, the Scott space $\Sigma P$ is a $d^{\ast}$-space if and only if it is $S^{\ast}$-well-filtered. Furthermore, some basic properties of $S^{\ast}$-well-filtered spaces are discussed. We prove that if $Y$ is an $S^{\ast}$-well-filtered space, the function space $TOP(X, Y)$ equipped with the Isbell topology may not be an $S^{\ast}$-well-filtered space. Finally, we study the $S^{\ast}$-well-filteredness of Smyth power spaces. In addition, Johnstone's non-sober dcpo example is shown to be $S^{\ast}$-well-filtered yet it is not strongly well-filtered, thereby establishing an obvious distinction between these two classes of dcpos.
\end{abstract}
\end{frontmatter}
\section{Introduction}\label{s1}
Domain theory serves as a foundational cornerstone in denotational semantics of programming languages (see \cite{s1982,g2003}). Sober spaces, well-filtered spaces and d-spaces as three of the most crucial classes of spaces in domain theory have been extensively studied by many scholars in various aspects. For example, in \cite{ll2021}, Liu, Li and Ho proved that for any nonempty $T_{0}$-space $X$, if the function space $[X, Y]$ of continuous functions from $X$ to $Y$ equipped with the Isbell topology is a d-space (resp., well-filtered space), then $Y$ is a d-space (resp., well-filtered space). In \cite{l2023}, Li et al. proved that for any nonempty $T_{0}$-space $X$, if $Y$ is a well-filtered space, then the function space $[X, Y]$ equipped with the Isbell topology is a well-filtered space. Meanwhile, they also showed that for any nonempty $T_{0}$-space $X$, if $Y$ and $\Sigma(\mathcal {O}(X))$ are sober spaces, then the function space $[X, Y]$ equipped with the Isbell topology is a sober space. In \cite{xs2020}, Xu et al. proved that a $T_{0}$-space $X$ is well-filtered iff $P_{S}(X)$ is well-filtered iff $P_{S}(X)$ is a d-space, where the Smyth power space $P_{S}(X)$ of a space $X$ is the set of all nonempty saturated compact sets equipped with the upper Vietoris topology.
 Particularly, some scholars have studied the three classes of spaces from the perspective of category theory and obtained that the category $\mathbf{Sob}$ (resp., $\mathbf{Top}_{w}$, $\mathbf{Top}_{d}$ ) of all sober spaces (resp., well-filtered spaces, d-spaces) with continuous mappings is reflective in the category $\mathbf{Top}_{0}$ of all $T_{0}$-spaces with continuous mappings (see \cite{w1981,w2019,k2009}).

It is well-known that Rudin's Lemma occupies a vital position in domain theory (see \cite{g2003,g2013}). Subsequently, a topological variant of Rudin's Lemma was provided by Heckmann and Keimel in \cite{h2013}. Inspired by the topological variant of Rudin's Lemma and Xi and Lawson's work \cite{x2017} on well-filtered spaces,  a series of characterizations of well-filtered spaces and sober spaces, and the concept of strong d-spaces were presented by Xu and Zhao in \cite{x2020}. Since then, some properties of the strong d-spaces have been discussed by many scholars \cite{l2023,x2021,x2020,al2024,bl2024}. Recently, in order to explore finer links between $T_{2}$-spaces and d-spaces, the definition of strongly well-filtered spaces was proposed by Xu in \cite{x2026}. It is closely related to well-filtered spaces and strong d-spaces.

In recent years, the investigation of some weakened versions of sober space and well-filtered space has become a significant research direction. For example, in \cite{l2017}, Lu and Li introduced the concept of weak well-filtered space and showed that Johnstone's example is weak well-filtered but not well-filtered. Furthermore, they proved that on weak well-filtered dcpos $P$, the coherence is equivalent to the compactness of ${\uparrow}x\cap{\uparrow}y$ for any $x, y\in P$, which generalized the key result in \cite{j2016}. In \cite{l2019}, Lu, Liu and Li presented the notion of weak sober space and some characterizations of weak sobriety were studied in \cite{l2021}. In \cite{c2023}, Chu and Li proposed the definition of $d^{\ast}$-spaces, which are strictly weaker than strong d-spaces. Moreover, some basic properties and characterizations of $d^{\ast}$-spaces are provided.
Inspired by these works, the main objective of this paper is to introduce a new class of topological spaces named $S^{\ast}$-well-filtered spaces, which are strictly weaker than strongly well-filtered spaces. Meanwhile, we investigate some of their basic properties and establish specific relationships among $T_{2}$-spaces, $S^{\ast}$-well-filtered spaces, $d^{\ast}$-spaces and weak well-filtered spaces, which will further enrich the theory of $T_{0}$ topological spaces.
The structures of this paper are organized as follows.

Section \ref{s2} provides some definitions and related results that are required for the remainder of the paper.

Section \ref{s3} reviews the concept and some important results of $d^{\ast}$-spaces. We provide several characterizations of $d^{\ast}$-spaces and show that the category $\mathbf{Top}_{d^{\ast}}$ of all $d^{\ast}$-spaces with continuous mappings is not reflective in
the category $\mathbf{Top}_{0}$ of all $T_{0}$-spaces with continuous mappings.

Section \ref{s4} proposes a new class of topological spaces named $S^{\ast}$-well-filtered spaces, which strictly contains all strongly well-filtered spaces. We establish specific connections among spaces lying between $T_{2}$-spaces and weak well-filtered spaces. It is demonstrated that for any dcpo $P$, the Scott space $\Sigma P$ is a $d^{\ast}$-space if and only if it is $S^{\ast}$-well-filtered. Moreover, we show that Johnstone's non-sober dcpo example is $S^{\ast}$-well-filtered yet it is not strongly well-filtered, thereby establishing a clear distinction between the two classes of dcpos.

Section \ref{s5} explores more properties of the $S^{\ast}$-well-filtered spaces. First, we show that the $S^{\ast}$-well-filteredness is closed-hereditary and saturated-hereditary. Second, we prove that if the product space of a family of $T_{0}$-space $X_{i}$ is an $S^{\ast}$-well-filtered space, then each space $X_{i}$ is an $S^{\ast}$-well-filtered space. Meanwhile, we give an example to show that the product of a family of $S^{\ast}$-well-filtered space is not $S^{\ast}$-well-filtered in general. Furthermore, we verify that if the function space $\mathrm{TOP}(X,Y)$ equipped with the Isbell topology is an $S^{\ast}$-well-filtered space for some nonempty space $X$, then $Y$ is an $S^{\ast}$-well-filtered space. However, the converse does not hold.

Section \ref{s6} discusses the $S^{\ast}$-well-filteredness of Smyth power spaces.
\section{Preliminaries}\label{s2}
In this section, we recall some fundamental concepts and results concerned to be used in this paper. For further details, we refer to \cite{g2003,g2013,a1994}.

Let $(P,\leq)$ be a poset. Given any subset $X$ of $P$, we denote that ${\downarrow} X=\{y\in P~|~y\leq x~\mathrm{for~some}~x\in X\}$ and ${\uparrow} X=\{y\in P~|~x\leq y~\mathrm{for~some}~x\in X\}$. For any element $x\in P$, we write ${\downarrow}x$ for ${\downarrow}\{x\}$ and ${\uparrow}x$ for ${\uparrow}\{x\}$. A subset $A$ of $P$ is called an upper set (resp., a lower set) if $A={\uparrow}A$ (resp., $A={\downarrow}A$). The poset $P$ is called a sup-complete poset if every nonempty subset $A$ has a supremum ${\bigvee}A$ in $P$. Let $P^{(<\omega)}=\{F\subseteq P\mid F~\mathrm{is~a~nonempty~finite ~subset}\}$ and $\mathbf{Fin}~P=\{{\uparrow}F\mid F\in P^{(<\omega)}\}$. The set of all natural numbers is denoted by $\mathbb{N}$ and let $\mathbb{N}^{+}=\mathbb{N}{\setminus}\{0\}$.

A nonempty subset $D$ of $P$ is called directed (resp., filtered) if for any $x, y \in D$, there is a $z\in D$ such that $x\leq z$ (resp., $z\leq x$) and $y\leq z$ (resp., $z\leq y$). The poset $P$ is a directed complete poset (dcpo, for short) if every directed subset of $P$ has a supremum. The set of all directed sets of $P$ is denoted by $\mathcal {D}(P)$.

For two elements $x, y \in P$, $x$ is way below $y$, denoted by $x\ll y$, if for all directed subsets $E\subseteq P$ for which ${\bigvee}E$ exists, the relation $y\leq {\bigvee}E$ always implies that there exists $e\in E$ with $x\leq e$. For any $x\in P$, if $x\ll x$, then $x$ is called a compact element of $P$. The set of compact elements of $P$ is denoted by $K(P)$. A poset $P$ is called algebraic if for any $x\in P$, $x=\bigvee({\downarrow}x \cap K(P))$, i.e., for all $x\in P$ the set ${\downarrow}x\cap K(P)$ is directed and $x$ is its supremum.
A poset $P$ is called a continuous poset if for every element $a\in P$, the set $\{x\in P \mid x\ll a\}$ is a directed set and
$a=\bigvee\{x\in P \mid x\ll a\}.$ A continuous dcpo is called a domain.

A subset $U$ of a poset $P$ is Scott open if (i) $U ={\uparrow}U$, and (ii) for any directed subset $D$ of $P$ with ${\bigvee}D$
existing,${\bigvee}D\in U$ implies $D\cap U\neq\emptyset$. All Scott open subsets of $P$ form a topology, which is called
the Scott topology on $P$ and denoted by $\sigma(P)$. The space $\Sigma P = (P,\sigma(P))$ is called the Scott space of $P$. A subset $A$ of $P$ is Scott closed if (i) $A={\downarrow}A$, and (ii) for any directed subset $D$ of $P$, $D\subseteq A$ implies ${\bigvee}D\in A$ if ${\bigvee}D$ exists.

In this paper, all topological spaces are assumed to be $T_{0}$. For any topological space $X$, we denote by $cl(Y)$ the closure of a subset $Y$ of $X$. The specialization order $\leq_{\tau}$ on $(X,\tau)$ is defined by $x\leq_{\tau} y$ iff $x\in cl_{\tau}(\{y\})$. Clearly, $cl_{\tau}(\{y\})={\downarrow}_{\tau}y$. Let
$\mathcal {O}(X)$ (resp., $\Gamma(X)$) be the set of all open subsets (resp., closed subsets) of $X$.

For any subset $Y$ of a space $X$, the saturation of $Y$, denoted by $\mathrm{sat}(Y)$, is defined as $$\mathrm{sat}(Y)=\bigcap\{U\in \mathcal {O}(X)~|~Y\subseteq U\},$$ or equivalently, $\mathrm{sat}(Y)={\uparrow}Y$ with respect to the specialization order. A subset $Y$ of $X$ is called saturated if $Y=\mathrm{sat}(Y)$.

For a topological space $X$, let $K(X)$ be the poset of all nonempty compact saturated subsets of $X$ endowed with the Smyth order, i.e., for $K_{1}, K_{2}\in K(X)$, $K_{1}\leq K_{2}$ iff $K_{2}\subseteq K_{1}$.
\begin{defn}\label{d1}(\cite{w1967}) A topological space is said to be KC if every compact (not necessarily $T_{2}$) subset is closed.
\end{defn}
It is easy to see that every $T_{2}$ space is KC and every KC space is $T_{1}$.
\begin{defn}\label{d2}(\cite{g2003, g2013}) Let $X$ be a topological space.
\begin{enumerate}[(1)]
\item A nonempty subset $A$ of $X$ is irreducible if for any $B, C\in \Gamma(X)$, $A\subseteq B\cup C$ implies $A\subseteq B$ or $A\subseteq C$. Denote by $Irr_{c}(X)$ the set of all irreducible closed set of $X$. A topological space $X$ is called sober, if for every $A\in Irr_{c}(X)$, there is a unique point $x\in X$ such that $A=cl\{x\}$.
\item A topological space $X$ is called well-filtered if for any filtered family $\{K_{i}\mid i\in I\}\subseteq K(X)$ and any open set $U$, $\bigcap_{i\in I} K_{i}\subseteq U$ implies $K_{i}\subseteq U$ for some $i\in I$.
\end{enumerate}
\end{defn}
\begin{defn}\label{d3}(\cite{l2017})
A topological space $X$ is called weak well-filtered if for any filtered family $\{K_{i}\mid i\in I\}\subseteq K(X)$ and $U\in \mathcal {O}(X){\setminus}\{\emptyset\}$, $\bigcap_{i\in I}K_{i}\subseteq U$ implies $K_{i}\subseteq U$ for some $i\in I$.
\end{defn}
\begin{defn}\label{d5}(\cite{x2020}) A $T_{0}$-space $X$ is called a strong d-space if for any $D\in \mathcal {D}(X)$, $x\in X$ and $U\in \mathcal {O}(X)$, $\bigcap_{d\in D}{\uparrow}d\cap {\uparrow}x\subseteq U$ implies ${\uparrow}d\cap{\uparrow}x\subseteq U$ for some $d\in D.$ The category
of all strong d-spaces with continuous mappings is denoted by $\mathbf{S}$-$\mathbf{Top}_{d}$.
\end{defn}
\begin{defn}\label{d6}(\cite{x2026}) A $T_{0}$-space $X$ is called strongly well-filtered if for any filtered family $\{K_{i}\mid i\in I\}\subseteq K(X)$, $G\in K(X)$ and $U\in \mathcal {O}(X)$, $\bigcap_{i\in I}K_{i}\cap G\subseteq U$ implies $K_{i}\cap G\subseteq U$ for some $i\in I$. The category of all strongly well-filtered spaces with continuous mappings is denoted by $\mathbf{S}$-$\mathbf{Top}_{w}$.
\end{defn}
In \cite{x2026}, Xu proved that every strongly well-filtered space is well-filtered and every strongly well-filtered space is a strong d-space.
\begin{defn}\label{d17}(\cite{g2003, w1981})
A space $X$ is called a $d$-space (or monotone convergence space) if $X$ (endowed with the specialization order) is a dcpo and $\mathcal {O}(X)\subseteq \sigma(X)$.
\end{defn}
Let $\mathbf{Top}_{0}$ be the category of all $T_{0}$-spaces with continuous mappings. In \cite{k2009}, Keimel and Lawson proved that a full subcategory $\mathbf{K}$ is reflective in $\mathbf{Top}_{0}$ if it satisfies the
following four conditions:
\begin{enumerate}[($K$1)]
  \item $\mathbf{K}$ contains all sober spaces;
  \item If $X\in \mathbf{K}$ and $Y$ is homeomorphic to $X$, then $Y\in \mathbf{K}$;
  \item If $\{X_{i}\mid i\in I\}\subseteq\mathbf{K}$ is a family of subspaces of a sober space, then the subspace $\bigcap_{i\in I}X_{i}\in\mathbf{K}$;
  \item If $f: X\rightarrow Y$ is a continuous map from a sober space $X$ to a sober space $Y$, then for any subspace $Z$ of $Y$, $Z\in \mathbf{K}$ implies that
$f^{-1}(Z)\in \mathbf{K}$.
\end{enumerate}

Let $X$ be a $T_{0}$-space. The b-topology associated with $X$ is the topology, which has the family $\{U \cap cl(\{x\})\mid x \in U\in O(X)\}$ as a base \cite{s1969}.
\begin{defn}\label{d18}(\cite{s2024}) For a subcategory $\mathbf{K}$ of $\mathbf{Top}_{0}$, we say that
\begin{enumerate}[(1)]
  \item $\mathbf{K}$ is productive if the product $\prod_{i\in I}X_{i}\in \mathbf{K}$ whenever $\{X_{i}\mid i\in I\}\subseteq \mathbf{K}$;
  \item  $\mathbf{K}$ is b-closed-hereditary, if $A\in \mathbf{K}$ whenever $A$ is a b-closed subspace of some $X\in \mathbf{K}$.
\end{enumerate}
\end{defn}
\begin{lem}\label{l20}(\cite{s2024})
 Let $\mathbf{K}$ be a full subcategory of $\mathbf{Top}_{0}$ with $\mathbf{K}\nsubseteq \mathbf{Top}_{1}$. Then $\mathbf{K}$ is reflective in $\mathbf{Top}_{0}$ if and only if $\mathbf{K}$ is productive and b-closed-hereditary.
\end{lem}
\section{The $d^{\ast}$-spaces}\label{s3}
The concept of $d^{\ast}$-spaces was introduced by Chu and Li in \cite{c2023}. For further details, we refer the readers to \cite{c2023}.
\begin{defn}\label{d11}(\cite{c2023})
A $T_{0}$-space $X$ is called a $d^{\ast}$-space if for any $D\in \mathcal {D}(X)$, $x\in X$ and $U\in \mathcal {O}(X){\setminus}\{\emptyset\}$, $\bigcap_{d\in D}{\uparrow}d\cap {\uparrow}x\subseteq U$ implies ${\uparrow}d\cap{\uparrow}x\subseteq U$ for some $d\in D.$ The category
of all $d^{\ast}$-spaces with continuous mappings is denoted by $\mathbf{Top}_{d^{\ast}}$.
\end{defn}

It is straightforward to verify the following results.
\begin{rem}\label{r1}
\begin{enumerate}[(1)]
 \item Every strong d-space is a $d^{\ast}$-space.
 \item Every $T_{1}$-space is a $d^{\ast}$-space.
 \item Every $KC$ space is a $d^{\ast}$-space.
 \item Every sup-complete poset (especially, every complete lattice) with Scott topology is a $d^{\ast}$-space.
 \end{enumerate}
\end{rem}
Similar to the characterization of strong d-spaces given in \cite{x2020}, we have the following results.
\begin{prop}\label{p1} For a $T_{0}$-space $X$, the following two conditions are equivalent:
\begin{enumerate}[(1)]
  \item\label{p1i1} $X$ is a $d^{\ast}$-space.
  \item\label{p1i2} For any $D\in \mathcal {D}(X)$, ${\uparrow}F\in \mathbf{Fin}~X$ and $U\in \mathcal {O}(X){\setminus}\{\emptyset\}$, $\bigcap_{d\in D}{\uparrow}d\cap {\uparrow}F\subseteq U$ implies ${\uparrow}d\cap{\uparrow}F\subseteq U$ for some $d\in D.$
\end{enumerate}
\end{prop}
\begin{proof}
\begin{description}
\item $(\ref{p1i1})\Rightarrow(\ref{p1i2})$ Suppose that $D\in \mathcal {D}(X)$, ${\uparrow}F\in \mathbf{Fin}~X$ and $U\in \mathcal {O}(X){\setminus}\{\emptyset\}$ satisfying $\bigcap_{d\in D}{\uparrow}d\cap {\uparrow}F\subseteq U$. Let $F=\{x_{1}, x_{2},..., x_{n}\}$. Then $\bigcap_{d\in D}{\uparrow}d\cap {\uparrow}F=\bigcup_{i=1}^{n}(\bigcap_{d\in D}{\uparrow}d\cap {\uparrow}x_{i})\subseteq U$. Thus, for any $x_{i}\in F$, $\bigcap_{d\in D}{\uparrow}d\cap {\uparrow}x_{i}\subseteq U$, and there exists $d_{x_{i}}\in D$ such that ${\uparrow}d_{x_{i}}\cap{\uparrow}x_{i}\subseteq U$ since $X$ is a $d^{\ast}$-space. As $D$ is directed and $F$ is finite, there exists $d_{0}\in D$ such that $\bigcup_{i=1}^{n}({\uparrow}d_{0}\cap{\uparrow}x_{i})\subseteq U$. Thus, we conclude that ${\uparrow}d_{0}\cap\bigcup_{i=1}^{n}{\uparrow}x_{i}={\uparrow}d_{0}\cap{\uparrow}F\subseteq U$ for some $d_{0}\in D$.
 \end{description}
\item $(\ref{p1i2})\Rightarrow(\ref{p1i1})$ Trivial.
\end{proof}
\begin{lem}\label{l12}(\cite{c2023}) For a dcpo $P$ , the following two conditions are equivalent:
\begin{enumerate}[(1)]
  \item $\Sigma P$ is a $d^{\ast}$-space.
  \item For any $C\in \Gamma(\Sigma P){\setminus}\{P\}$ and $x\in P$, ${\downarrow}({\uparrow}x\cap C)\in \Gamma(\Sigma P)$.
\end{enumerate}
\end{lem}
\begin{lem}\label{l14}(\cite{h2013}) Let $P$ be a dcpo, $U$ a Scott-open set of $P$ and $\{{\uparrow}F_{i}\mid i\in I\}\subseteq \mathbf{Fin}~P$ a filtered family with $\bigcap_{i\in I}{\uparrow}F_{i}\subseteq U.$ Then ${\uparrow}F_{i}\subseteq U$ some $i\in I$.
\end{lem}
\begin{prop}\label{p18}
For a dcpo $P$, the following two conditions are equivalent:
\begin{enumerate}[(1)]
  \item $\Sigma P$ is a $d^{\ast}$-space.
  \item For any filtered family $\{{\uparrow}F_{d}\mid d\in D\}\subseteq \mathbf{Fin}~P$, ${\uparrow}F\in \mathbf{Fin}~P$ and $U\in \sigma(P){\setminus}\{\emptyset\}$, $\bigcap_{d\in D}{\uparrow}F_{d}\cap{\uparrow}F\subseteq U$ implies ${\uparrow}F_{d}\cap{\uparrow}F\subseteq U$ some $d\in D$.
\end{enumerate}
\end{prop}
\begin{proof}
\begin{description}
\item $(\ref{p1i1})\Rightarrow(\ref{p1i2})$ Assume that $\{{\uparrow}F_{d}\mid d\in D\}\subseteq \mathbf{Fin}~P$, ${\uparrow}F\in \mathbf{Fin}~P$ and $U\in \sigma(P){\setminus}\{\emptyset\}$ satisfying $\bigcap_{d\in D}{\uparrow}F_{d}\cap{\uparrow}F\subseteq U$. Then $\bigcap_{d\in D}{\uparrow}F_{d}\cap{\downarrow}({\uparrow}F\cap(X\setminus U))=\emptyset$. Therefore, $\bigcap_{d\in D}{\uparrow}F_{d}\subseteq X\setminus{\downarrow}({\uparrow}F\cap(X\setminus U))$ and $X\setminus{\downarrow}({\uparrow}F\cap(X\setminus U))\in \sigma(P){\setminus}\{\emptyset\}$ by Lemma \ref{l12}. It follows from Lemma \ref{l14} that ${\uparrow}F_{d}\subseteq X\setminus{\downarrow}({\uparrow}F\cap(X\setminus U))$ for some $d\in D$. So ${\uparrow}F_{d}\cap{\uparrow}F\subseteq U$ for some $d\in D$.
 \end{description}
\item $(\ref{p1i2})\Rightarrow(\ref{p1i1})$ Trivial.
\end{proof}
One of the most fundamental properties of $d^{\ast}$-spaces is the following result.
\begin{lem}\label{l1}(\cite{c2023}) A retract of a $d^{\ast}$-space is a $d^{\ast}$-space.
\end{lem}
\begin{prop}\label{p2}
Let $\{X_{i}\mid i\in I \}$ be a family of $T_{0}$-space. If the product space $\prod _{i\in I}X_{i}$ is a $d^{\ast}$-space, then $X_{i}$ is a $d^{\ast}$-space for each $i\in I$.
\end{prop}
\begin{proof}
Let $X=\prod _{i\in I}X_{i}$ be a $d^{\ast}$-space. For each $i\in I$, $X_{i}$ is a retract of $X$. It follows from Lemma \ref{l1} that $X_{i}$ is a $d^{\ast}$-space for each $i\in I$.
\end{proof}
Conversely, one naturally asks the question: is the product space of an arbitrary family of $d^{\ast}$-spaces again a $d^{\ast}$-space?

In \cite{c2023}, the authors gave an example to reveal that the product space of two $d^{\ast}$-spaces is not a $d^{\ast}$-space. Hence, the product space of an arbitrary family of $d^{\ast}$-spaces is not a $d^{\ast}$-space.

In \cite{l2023}, we know that the category $\mathbf{S}$-$\mathbf{Top}_{d}$ is not a reflective subcategory of $\mathbf{Top}_{0}$. By Lemma \ref{l20}, we can immediately get the following result.
\begin{thm}\label{p3}
The category $\mathbf{Top}_{d^{\ast}}$ is not a reflective subcategory of $\mathbf{Top}_{0}$.
\end{thm}
\section{$S^{\ast}$-well-filtered spaces}\label{s4}
In this section, we introduce the $S^{\ast}$-well-filtered space inspired by Lu and Li's work \cite{l2017} on weak well-filtered spaces.
\begin{defn}\label{d12} A $T_{0}$-space $X$ is called $S^{\ast}$-well-filtered if for any filtered family $\{K_{i}\mid i\in I\}\subseteq K(X)$, $G\in K(X)$ and $U\in \mathcal {O}(X){\setminus}\{\emptyset\}$, $\bigcap_{i\in I}K_{i}\cap G\subseteq U$ implies $K_{i}\cap G\subseteq U$ for some $i\in I$. The category
of all $S^{\ast}$-well-filtered spaces with continuous mappings is denoted by $\mathbf{S^{\ast}}$-$\mathbf{Top}_{w}$.
\end{defn}
Obviously, every strongly well-filtered space is $S^{\ast}$-well-filtered. The next example shows that an $S^{\ast}$-well-filtered space may not be a strongly well-filtered space in general.
\begin{exmp}\label{e1}
Consider the space $\Sigma\mathbb{N}=(\mathbb{N},\sigma(\mathbb{N}))$, where $\mathbb{N}$ is the set of natural numbers equipped with the Scott topology under the usual ordering. Then\\
(a) $\Sigma\mathbb{N}$ is an $S^{\ast}$-well-filtered space.

Assume that $\{{\uparrow}n\mid n\in \mathbb{N}\}\subseteq K(\Sigma\mathbb{N})$ is a filtered family, ${\uparrow} m\in K(\Sigma\mathbb{N})$ and $U\in \mathcal {O}(\Sigma\mathbb{N}){\setminus}\{\emptyset\}$ satisfying $\bigcap_{n\in \mathbb{N}}{\uparrow}n\cap {\uparrow}m\subseteq U$. Then $D=\{n\mid n\in \mathbb{N}\}$ is a directed set. If $D$ is finite, then $\bigcap_{n\in D}{\uparrow}n\cap {\uparrow}m={\uparrow}\mathrm{sup}D\cap {\uparrow}m\subseteq U$. If $D$ is infinite, then $\bigcap_{n\in D}{\uparrow}n=\emptyset$. So $\bigcap_{n\in D}{\uparrow}n\cap {\uparrow}m\subseteq U$. Since $U\neq\emptyset$, there exists $k\in \mathbb{N}$ such that $U={\uparrow}k$. Thus, there exists $l\in D$ such that $k\leq l$. So we conclude that ${\uparrow}l\cap {\uparrow}m\subseteq U$ for some $l\in D$.\\
(b) $\Sigma\mathbb{N}$ is not a strongly well-filtered space.

Indeed, $\bigcap_{n\in \mathbb{N}}{\uparrow}n=\emptyset$. But for any $n\in \mathbb{N}$, ${\uparrow}n\neq\emptyset$.
It implies that $\Sigma\mathbb{N}$ is not well-filtered. Hence, it is not strongly well-filtered.
\end{exmp}
Example \ref{e1} also shows that an $S^{\ast}$-well-filtered space may not be well-filtered.
\begin{prop}\label{p5}
A $T_{0}$-space $X$ is $S^{\ast}$-well-filtered if and only if for each filtered family $\{K_{i}\mid i\in I\}\subseteq K(X)$, $G\in K(X)$ and $C\in \Gamma(X){\setminus}\{X\}$, $K_{i}\cap G\cap C\neq \emptyset$ for each $i\in I$ implies $\bigcap_{i\in I}K_{i}\cap G\cap C\neq \emptyset$.
\end{prop}
A poset $P$ is called $S^{\ast}$-well-filtered (resp., strongly well-filtered, weak well-filtered) if $P$ with its Scott topology $\sigma(P)$ is a $S^{\ast}$-well-filtered (resp., strongly well-filtered, weak well-filtered) space.
\begin{prop}\label{p8}
Let $P$ be a dcpo which has a largest element $\top$. Then $P$ is strongly well-filtered if and only if it is $S^{\ast}$-well-filtered.
\end{prop}
\begin{proof}
If $P$ is strongly well-filtered, then it is $S^{\ast}$-well-filtered. Conversely, assume that $\{K_{i}\mid i\in I\}\subseteq K(\Sigma P)$ is a filtered family, $G\in K(\Sigma P)$ and $U\in \mathcal {O}(X)$ satisfying $\bigcap_{i\in I}K_{i}\cap G\subseteq U$. Then $\top\in \bigcap_{i\in I}K_{i}\cap G$, that is, $\bigcap_{i\in I}K_{i}\cap G\neq\emptyset$. It follows that $U\neq\emptyset$. Thus, there exists $i\in I$ such that $K_{i}\cap G\subseteq U$ since $P$ is $S^{\ast}$-well-filtered. It implies that $P$ is strongly well-filtered.
\end{proof}
The space $X$ is called coherent if the intersection of any two compact saturated sets is again compact. Clearly, every KC space is coherent.
\begin{prop}\label{p7}
Let $X$ be a $T_{0}$-space. Consider the following conditions:
\begin{enumerate}[(1)]
\item\label{p7i1} $X$ is $S^{\ast}$-well-filtered;
\item\label{p7i2} $X$ is weak well-filtered.
\end{enumerate}
Then $(\ref{p7i1})\Rightarrow(\ref{p7i2})$. Moreover, $(\ref{p7i1})\Leftrightarrow(\ref{p7i2})$ if $X$ is coherent.
\end{prop}
\begin{proof}
$(\ref{p7i1})\Rightarrow(\ref{p7i2})$ Let $X$ be an $S^{\ast}$-well-filtered space. Assume that $\{K_{i}\mid i\in I\}\subseteq K(X)$ is a filtered family and $U\in \mathcal {O}(X){\setminus}\{\emptyset\}$ satisfying $\bigcap_{i\in I} K_{i}\subseteq U$. Select an $i_{0}\in I$, let $G=K_{i_{0}}$. Then $\bigcap_{i\in I} K_{i}\cap G=\bigcap_{i\in I} K_{i}\subseteq U$. Thus, there exists $i_{1}\in I$ such that $K_{i_{1}}\cap K_{i_{0}}\subseteq U$ by the $S^{\ast}$-well-filteredness of $X$. So there exists $i_{2}\in I$ such that $K_{i_{2}}\subseteq K_{i_{1}}\cap K_{i_{0}}\subseteq U$. It implies that $X$ is a weak well-filtered space.

$(\ref{p7i2})\Rightarrow(\ref{p7i1})$ Let $X$ be a coherent weak well-filtered space. Assume that $\{K_{i}\mid i\in I\}\subseteq K(X)$ is a filtered family, $G\in K(X)$ and $U\in \mathcal {O}(X){\setminus}\{\emptyset\}$ satisfying $\bigcap_{i\in I} K_{i}\cap G\subseteq U$. If there exists $i_{0}\in I$ such that $K_{i_{0}}\cap G=\emptyset$, then $K_{i_{0}}\cap G\subseteq U$. Now assume that $K_{i}\cap G\neq\emptyset$ for all $i\in I$. Then for all $i\in I$, $K_{i}\cap G\in K(X)$ by the coherence of $X$, $\{K_{i}\cap G\mid i\in I\}\subseteq K(X)$ is a filtered family. It follows that $\bigcap_{i\in I}(K_{i}\cap G)=\bigcap_{i\in I} K_{i}\cap G\subseteq U$. Because $X$ is a weak well-filtered space, there exists $i_{1}\in I$ such that $K_{1}\cap G\subseteq U$. So $X$ is an $S^{\ast}$-well-filtered space.
\end{proof}
By Proposition \ref{p7}, we get the following corollary.
\begin{cor}\label{c1} Let $X$ be a KC space. If $X$ is weak well-filtered, then it is $S^{\ast}$-well-filtered.
\end{cor}
\begin{defn}\label{d13}(\cite{x2000})
A poset $P$ is called a consistent dcpo if for any directed subset $D$ of $P$ with $\bigcap_{d\in D}{\uparrow}d\neq\emptyset$, $D$ has a least upper bound in $P$.
\end{defn}
\begin{lem}\label{l3}(\cite{l2017}) Let $(X,\tau)$ be a weak well-filtered space. Then $\Omega(X)$ is a consistent dcpo and $\tau\subseteq\sigma(\Omega(X))$, where $\Omega(X)=(X, \leq_{\tau})$.
\end{lem}
\begin{cor}\label{c7}
Let $(X,\tau)$ be an $S^{\ast}$-well-filtered space. Then $\Omega(X)$ is a consistent dcpo and $\tau\subseteq\sigma(\Omega(X))$, where $\Omega(X)=(X, \leq_{\tau})$.
\end{cor}
The following example illustrates that a non-coherent weak well-filtered space may not be an $S^{\ast}$-well-filtered space.
\begin{exmp}\label{e4}
Let $P=\mathbb{N}^{+}\cup\{\omega, a\}\cup\{\omega_{n}\mid n\in \mathbb{N}^{+}\}.$ Define an order $\leq$ on $P$ as follows:\\$x\leq y$ if and only if
\begin{enumerate}[(1)]
 \item $x\leq y~in~\mathbb{N}^{+}$; or
 \item $x\in \mathbb{N}^{+}\cup\{\omega\}, y=\omega$; or
 \item $x=a, y\in \{\omega_{n}\mid n\in \mathbb{N}^{+}\}$; or
 \item $x\in \mathbb{N}^{+}, y\in \{\omega_{n}\mid n\geq x\}$.
  \end{enumerate}
  \begin{figure}[h]
		\centering
		\begin{tikzpicture}[scale=0.5]
			\path (-8,-6) node[left]{$1$} coordinate;
			\fill (-8,-6) circle (2pt);
			\path (-8,-4) node[left]{$2$} coordinate;
			\fill (-8,-4) circle (2pt);
			\path (-8,-2) node[left]{$3$} coordinate;
			\fill (-8,-2) circle (2pt);
            \path (-8,1) node[left]{$n$} coordinate;
			\fill (-8,1) circle (2pt);
			\path (-7.5,5) node[left]{$\omega$} coordinate;
			\fill (-8,4) circle (2pt);
			\path (-5.5,5) node[left]{$\omega_1$} coordinate;
			\fill (-6,4) circle (2pt);
			\path (-3.5,5) node[left]{$\omega_2$} coordinate;
			\fill (-4,4) circle (2pt);
			\path (-1.5,5) node[left]{$\omega_3$} coordinate;
			\fill (-2,4) circle (2pt);
            \path (1,5) node[left]{$\omega_n$} coordinate;
			\fill (0.5,4) circle (2pt);
			\path (-2.5,0.5) node[left]{$a$} coordinate;
			\fill (-3,1) circle (2pt);
			
			\draw (-3,1) -- (-6,4);
			\draw (-3,1) -- (-4,4);
			\draw (-3,1) -- (-2,4);
			\draw (-8,-6) -- (-8,-2);
			\draw (-8,-6) -- (-6,4);
			\draw (-8,-4) -- (-4,4);
			\draw (-8,-2) -- (-2,4);
             \draw (-8,1) -- (0.5,4);
             \draw (-3,1) -- (0.5,4);
			\draw[style=dashed] (-8,-1.8) -- (-8,3.8);
            \draw[style=dashed] (-1.8,4) -- (0.3,4);
			\draw[style=dashed] (0.8,4) -- (3,4);		
			
		\end{tikzpicture}
		\caption{The poset $P$ in Example \ref{e4}.}\label{fig1}
	\end{figure}
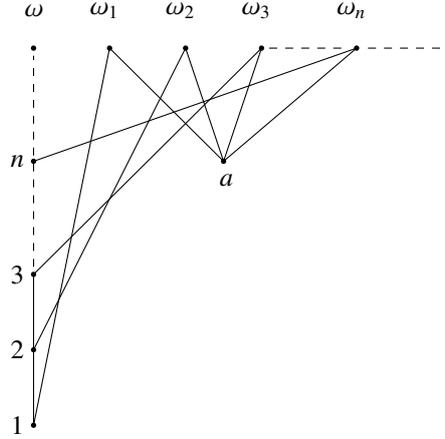

Clearly, $P$ is a dcpo and $x\ll x$ for all $x\in P{\setminus}\{\omega\}$. Then $P$ is an algebraic domain. Thus, $\Sigma P$ is sober; hence it is well-filtered. It follows that $\Sigma P$ is a weak well-filtered space. But for ${\uparrow}2$, ${\uparrow}a\in K(\Sigma P)$, obviously, ${\uparrow}2\cap{\uparrow}a=\{\omega_{2}, \omega_{3},...,\omega_{n},...\}$ is not compact. So $\Sigma P$ is not coherent. Given $m\in \mathbb{N}^{+}$, $\bigcap_{n\in \mathbb{N}^{+}}{\uparrow}n\cap{\uparrow}a=\emptyset\subseteq\{\omega_{m}\}\in \sigma(P),$ but ${\uparrow}n\cap{\uparrow}a=\{\omega_{n}, \omega_{n+1},...\}\nsubseteq\{\omega_{m}\}$ for any $n\in \mathbb{N}^{+}$. Therefore, $\Sigma P$ is not an $S^{\ast}$-well-filtered space.
\end{exmp}
Example \ref{e4} also indicates that a well-filtered space may not be $S^{\ast}$-well-filtered.
\begin{prop}\label{p8}
\begin{enumerate}[(1)]
 \item\label{p8i1} Every $T_{2}$-space is an $S^{\ast}$-well-filtered space.
 \item\label{p8i2}Every $S^{\ast}$-well-filtered space is a $d^{\ast}$-space.
 \end{enumerate}
\end{prop}
\begin{proof}
$(\ref{p8i1})$ Let $X$ be a $T_{2}$-space. Then $X$ is weak well-filtered and $K(X)\subseteq \Gamma(X){\setminus}\{\emptyset\}$. For any $K_{1}, K_{2}\in K(X)$, one has that $K_{1}\cap K_{2}$ is a compact subset of $X$. Thus, $X$ is coherent. It follows from Proposition \ref{p7} that $X$ is an $S^{\ast}$-well-filtered space.

$(\ref{p8i2})$ Let $X$ be an $S^{\ast}$-well-filtered space. Assume that $D\in \mathcal {D}(X)$, $x\in X$ and $U\in \mathcal {O}(X){\setminus}\{\emptyset\}$ satisfying $\bigcap_{d\in D}{\uparrow}d\cap {\uparrow}x\subseteq U$. Then $\{{\uparrow}d\mid d\in D\}\subseteq K(X)$ is a filtered family and ${\uparrow}x\in K(X)$. Since $X$ is an $S^{\ast}$-well-filtered space, there exists $d_{0}\in D$ such that ${\uparrow}d_{0}\cap {\uparrow}x\subseteq U$. So $X$ is a $d^{\ast}$-space.
\end{proof}
By Propositions \ref{p7} and \ref{p8}, one obtains the following result.
\begin{cor}\label{c2}Every coherent weak well-filtered space is a $d^{\ast}$-space.
\end{cor}
The following example illustrates that a $d^{\ast}$-space may not be an $S^{\ast}$-well-filtered space in general.
\begin{exmp}\label{e2}
Let $X_{1}=(\mathbb{N},\tau_{cof})$ be the space of the set of natural numbers $\mathbb{N}$ equipped with the co-finite topology $\tau_{cof}=\{\emptyset\}\cup\{U\subseteq X\mid X{\setminus}U~\mathrm{is~a~finite~set}\}$ and $X_{2}=(\{a\},\tau_{2})$ the space of the singleton set $\{a\}$ equipped with the discrete topology $\tau_{2}=\{\emptyset,\{a\}\}$. Consider the space $X=(X_{1}\cup X_{2},\tau)$, where $\tau$ is the common refinement $\tau_{cof}\vee\tau_{2}$ of the co-finite topology and the discrete topology. Then $X$ is a $T_{1}$-space since every singleton set is closed. Hence, it is a $d^{\ast}$-space. Next, we claim that $X$ is not an $S^{\ast}$-well-filtered space. In fact, $K(X)=2^{X}{\setminus}\{\emptyset\}$. Let $\mathcal {K}=\{X{\setminus}F\mid F\in X^{(<\omega)}\}$. Then $\mathcal {K}\subseteq K(X)$ is a filtered family. Clearly, $\bigcap\mathcal {K}=\emptyset\subseteq \{a\}$. But for all $F\in X^{(<\omega)}$, $X{\setminus}F\nsubseteq\{a\}$. Thus, $X$ is not weak well filtered. So it is not an $S^{\ast}$-well-filtered space by Proposition \ref{p7}.
\end{exmp}
 Example \ref{e2} also indicates that a $d^{\ast}$-space may not be a weak well-filtered space in general. The next example shows that there is an $S^{\ast}$-well-filtered space which is not a sober space.
 \begin{exmp}\label{e3}
Let $X=(R, \tau_{coc})$ be the space of the set of real numbers $R$ equipped with the co-countable topology $\tau_{coc}=\{\emptyset\}\cup\{U\subseteq X\mid X{\setminus}U~\mathrm{is~a~countable~set}\}$. Then $Irr_{c}(X)=\{X\}\cup\{\{x\}\mid x\in X\}$ and $K(X)=\{F\subseteq X\mid\mathrm{F~is~a~nonempty~finite~set}\}$. Obviously, $X$ is
a $T_{1}$-space but not a sober space. Assume that $\mathcal {K}=\{F_{i}\mid i\in I\}\subseteq K(X)$ is a filtered family and $U\in \mathcal {O}(X){\setminus}\{\emptyset\}$ satisfying $\bigcap\mathcal {K}\subseteq U$. Then $\mathcal {K}$ has a least element $F_{i_{0}}$ since $\mathcal {K}$ is filtered and all $F_{i}$ are finite. Therefore, $F_{i_{0}}=\bigcap\mathcal {K}\subseteq U$, that is, $X$ is a weak well-filtered space. As $X$ is coherent, it follows from Proposition \ref{p7} that $X$ is an $S^{\ast}$-well-filtered space.
\end{exmp}
Fig. \ref{fig2} illustrates the relationships among spaces lying between $T_{2}$-space and weak well-filtered space.
\begin{figure}[h]
	 	\centering
	 	\begin{tikzpicture}[scale=0.5]
	 		
	 		\path (10,2) node{weak well-filtered space} coordinate (a);
	 		\path (10,-1) node{$d^{\ast}$-space} coordinate (b);
	 		
            \path (1,5) node{well-filtered space} coordinate (c);
	 		\path (1,2) node{$S^{\ast}$-well-filtered space} coordinate (d);
	 		\path (1,-1) node{strong d-space} coordinate (e);
	 		
	 		\path (-8,5) node{sober space} coordinate (f);
	 		\path (-8,2) node{strongly well-filtered space} coordinate (g);
	 		\path (-8,-1) node{$T_1$-space} coordinate (h);
	 		
	 		\path (-15,2) node{$T_2$-space} coordinate (i);
	 		\path (-15,-1) node{KC space} coordinate (j);
	 		\path (11.5,0.8) node{coherent} coordinate (h);
	 	
	 		\draw [->] (-15,1.5) -- (-15,-0.5);
\draw [->] (-13.5,2) -- (-12,2);
	 		\draw [->] (-3.8,2) -- (-2.2,2);
	 		\draw [->] (4.6,2) -- (6.2,2);
	 		\draw [->] (3.8,-1) -- (8.4,-1);
	 		\draw [->] (-6.3,-1) -- (-1.4,-1);
	 		\draw [->] (-13.3,-1) -- (-9.5,-1);
	 		\draw [->] (-15,2.3) -- (-9.5,4.5);
	 		\draw [->] (-6,5) -- (-2,5);
	 		\draw [->] (3,4.5) -- (10,2.6);
	 		\draw [->] (-6,1.5) -- (1,-0.5);
	 		\draw [->] (3,1.5) -- (9.8,-0.5);
	 		\draw [->] (10,1.5) -- (10,-0.3);
	 		\draw [->] (-6,2.5) -- (1,4.5);
	 		
	 	\end{tikzpicture}
	 	\caption{The relationships of among spaces lying between $T_{2}$-space and weak well-filtered space.}\label{fig2}
	 \end{figure}
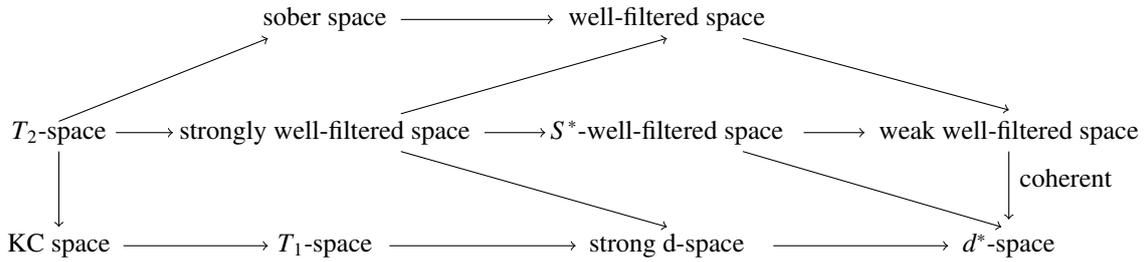
\begin{lem}\label{l11}(\cite{x2026})
Let $X$ be a $T_{0}$-space such that $\mathrm{max}(C)\neq\emptyset$ and ${\downarrow}(C\cap K)$ is closed for any nonempty $C\in \Gamma(X)$ and $K\in K(X)$, then $X$ is strongly well-filtered.
\end{lem}
By Lemma \ref{l11}, we have the following result.
\begin{prop}\label{p9}
Let $X$ be a $T_{0}$-space such that $\mathrm{max}(C)\neq\emptyset$ and ${\downarrow}(C\cap K)$ is closed for any nonempty $C\in \Gamma(X){\setminus}\{X\}$ and $K\in K(X)$, then $X$ is $S^{\ast}$-well-filtered.
\end{prop}
\begin{lem}\label{l13}(\cite{x2026}) Let $X$ be a $d$-space such that ${\downarrow}(C\cap K)$ is closed for any nonempty $C\in \Gamma(X)$ and $K\in K(X)$. Then $X$ is strongly well-filtered.
\end{lem}
The following result follows directly from Lemma \ref{l13}.
\begin{prop}\label{p16}
Let $X$ be a $d$-space such that ${\downarrow}(C\cap K)$ is closed for any nonempty $C\in \Gamma(X){\setminus}\{X\}$ and $K\in K(X)$.
Then $X$ is $S^{\ast}$-well-filtered.
\end{prop}
\begin{lem}\label{l6}(\cite{x2026}) For a dcpo $P$ and $C\subseteq P$, the following two conditions are equivalent:
\begin{enumerate}[(1)]
  \item ${\downarrow}({\uparrow}x\cap C)\in \Gamma(\Sigma P)$ for all $x\in P$;
  \item ${\downarrow}(K\cap C)=\bigcup_{k\in K}{\downarrow}({\uparrow}k\cap C)\in \Gamma(\Sigma P)$ for all $K\in K(\Sigma P)$.
\end{enumerate}
\end{lem}
The following statement is a direct consequence of Lemmas \ref{l12}, \ref{l6} and Propositions \ref{p8}, \ref{p9}.
\begin{thm}\label{t1} For a dcpo $P$ , the following four conditions are equivalent:
\begin{enumerate}[(1)]
  \item\label{t1i1} $\Sigma P$ is a $d^{\ast}$-space;
  \item\label{t1i2} For any $x\in P$ and $C\in \Gamma(\Sigma P){\setminus}\{P\}$, ${\downarrow}({\uparrow}x\cap C)\in \Gamma(\Sigma P)$;
  \item\label{t1i3} For any $K\in K(\Sigma P)$ and $C\in \Gamma(\Sigma P){\setminus}\{P\}$, ${\downarrow}(K\cap C)\in \Gamma(\Sigma P)$;
  \item\label{t1i4} $\Sigma P$ is an $S^{\ast}$-well-filtered space.
\end{enumerate}
\end{thm}
\begin{lem}\label{l14}(\cite{l2017}) Let $P$ be a weak well-filtered poset. Then $P$ is coherent if and only if ${\uparrow}x\cap{\uparrow}y$ is compact for all $x, y\in P$.
\end{lem}
The next example illustrates that $S^{\ast}$-well-filtered dcpos need not be strongly well-filtered.
\begin{exmp}\label{e9}
Consider that the dcpo $\mathbb{J}=\mathbb{N}\times(\mathbb{N}\cup\{\omega\})$ constructed by Johnstone in \cite{j1981}, with the order defined by $(j,k)\leq(m,n)$ if and only if $j=m$ and $k\leq n$, or $n=\omega$ and $k\leq m$. Then $\mathbb{J}$ is a dcpo. For each $(x,y)\in \mathbb{N}\times\mathbb{N}$, $${\uparrow}(x,y)\cap(\mathbb{N}\times\{\omega\})
=\{(x,\omega)\}\cup\{(z,\omega)\mid z\in\mathbb{N}~and~y\leq z\}.$$
Hence, $(\mathbb{N}\times\{\omega\}){\setminus}{\uparrow}(x,y)$ is finite. For any $U\in \sigma(\mathbb{J}){\setminus}\{\emptyset\}$,
there exists $(x,y)\in \mathbb{N}\times\mathbb{N}$ such that $(x,y)\in U$.
So $(\mathbb{N}\times\{\omega\}){\setminus}U$ is finite.

Assume that $K\subseteq\mathbb{N}\times\{\omega\}$ is a nonempty subset and $\{U_{i}\}_{i\in I}\subseteq\sigma(\mathbb{J})$ is a directed open cover of $K$ with $K\subseteq\bigcup_{i\in I}U_{i}$. For any $(x,\omega)\in K$, there exists $j\in I$ such that $(x,\omega)\in U_{j}$. Since $(\mathbb{N}\times\{\omega\}){\setminus}U_{j}$ is finite and $K{\setminus}U_{j}\subseteq(\mathbb{N}\times\{\omega\}){\setminus}U_{j}$, we have that $K{\setminus}U_{j}$ is finite. Therefore, there exists $k\in I$ such that $K{\setminus}U_{j}\subseteq U_{k}$. For $U_{j}$ and $U_{k}$, there is $l\in I$ such that $U_{j}\cup U_{k}\subseteq U_{l}$. Thus, $K=(K\cap U_{j})\cup(K{\setminus}U_{j})\subseteq U_{j}\cup U_{k}\subseteq U_{l}$. That is $K\in K(\Sigma\mathbb{J})$.

It is well-known that $\Sigma\mathbb{J}$ is a weak well-filtered space \cite[Example 3.1]{l2017}. Next, we claim that $\Sigma\mathbb{J}$ is coherent. For any $(x_{1}, y_{1}), (x_{2}, y_{2})\in \mathbb{J}$, we claim that ${\uparrow}(x_{1}, y_{1})\cap{\uparrow}(x_{2}, y_{2})$ is compact. If $x_{1}=x_{2}$, without loss of generality, assume that $y_{1}\leq y_{2}$, then $(x_{1}, y_{1})\leq (x_{2}, y_{2})$ and ${\uparrow}(x_{1}, y_{1})\cap{\uparrow}(x_{2}, y_{2})={\uparrow}(x_{2}, y_{2})$ is compact. If $x_{1}\neq x_{2}$, then ${\uparrow}(x_{1}, y_{1})\cap{\uparrow}(x_{2}, y_{2})=\{(m,\omega)\mid m\geq \mathrm{max}\{y_{1},y_{2}\}\}\subseteq\mathbb{N}\times\{\omega\}$. Thus, ${\uparrow}(x_{1}, y_{1})\cap{\uparrow}(x_{2}, y_{2})$ is compact by the previous paragraphs. So $\Sigma\mathbb{J}$ is coherent by Lemma \ref{l14}. It follows from Proposition \ref{p7} that $\Sigma\mathbb{J}$ is an $S^{\ast}$-well-filtered space.

Now we claim that $\Sigma\mathbb{J}$ is not a strongly well-filtered space. Clearly, $\bigcap_{n\in \mathbb{N}}{\uparrow}(2,n)\cap{\uparrow}(3,3)=\emptyset$, but ${\uparrow}(2,n)\cap{\uparrow}(3,3)=\{(m,\omega)\mid n\leq m\}\neq\emptyset$ for any $n\geq 3$. Thus, $\Sigma\mathbb{J}$ is not a strong d-space; hence, it cannot be strongly well-filtered.
\end{exmp}
\section{$\mathrm{Some~basic~properties~of}~S^{\ast}$-well-filtered spaces}\label{s5}
In this section, we shall explore more properties of the $S^{\ast}$-well-filtered space. First, we show that the $S^{\ast}$-well-filteredness is closed-hereditary and saturated-hereditary.
\begin{prop}\label{p10} A nonempty closed subspace of $S^{\ast}$-well-filtered space is an $S^{\ast}$-well-filtered space.
\end{prop}
\begin{proof}
Let $X$ be an $S^{\ast}$-well-filtered space and $Y$ a nonempty closed subspace of $X$. Assume that $\{K_{i}\mid i\in I\}\subseteq K(Y)$ is a filtered family, $G\in K(Y)$ and $C\in \Gamma(Y){\setminus}\{Y\}$ satisfying $K_{i}\cap G\cap C\neq\emptyset$ for all $i\in I$. Then $\{{\uparrow_{X}}K_{i}\mid i\in I\}\subseteq K(X)$ is a filtered family, ${\uparrow_{X}}G\in K(X)$ and $C\in \Gamma(X){\setminus}\{X\}$. Clearly, ${\uparrow_{X}}G\cap Y=G$ and ${\uparrow_{X}}K_{i}\cap G=K_{i}$ for each $i\in I$. Hence, ${\uparrow_{X}}K_{i}\cap{\uparrow_{X}}G\cap C={\uparrow_{X}}K_{i}\cap{\uparrow_{X}}G\cap C\cap Y=K_{i}\cap G\cap C\neq\emptyset$. Because $X$ is an $S^{\ast}$-well-filtered space, it follows from Proposition \ref{p5} that $\bigcap_{i\in I}{\uparrow_{X}}K_{i}\cap{\uparrow_{X}}G\cap C\neq\emptyset$. So $\bigcap_{i\in I}K_{i}\cap G\cap C=\bigcap_{i\in I}({\uparrow_{X}}K_{i}\cap Y)\cap({\uparrow_{X}}G\cap Y)\cap (C\cap Y)=\bigcap_{i\in I}{\uparrow_{X}}K_{i}\cap{\uparrow_{X}}G\cap C\neq\emptyset$. It implies that $Y$ is an $S^{\ast}$-well-filtered space.
\end{proof}
\begin{prop}\label{p11} A nonempty saturated subspace of $S^{\ast}$-well-filtered space is an $S^{\ast}$-well-filtered space.
\end{prop}
\begin{proof}
Let $X$ be an $S^{\ast}$-well-filtered space and $Y$ a nonempty saturated subspace of $X$. Assume that $\{K_{i}\mid i\in I\}\subseteq K(Y)$ is a filtered family, $G\in K(Y)$ and $U\in \mathcal {O}(Y){\setminus}\{\emptyset\}$ satisfying $\bigcap_{i\in I} K_{i}\cap G\subseteq U$. Then there exists $V\in \mathcal {O}(X){\setminus}\{\emptyset\}$ such that $U=V\cap Y$. As $Y={\uparrow_{X}}Y$, we obtain that $\{K_{i}\mid i\in I\}\subseteq K(X)$ is a filtered family, $G\in K(X)$. Thus, $\bigcap_{i\in I} K_{i}\cap G\subseteq U\subseteq V$. By the $S^{\ast}$-well-filteredness of $X$, there is $i_{0}\in I$ such that $K_{i_{0}}\cap G\subseteq V$. It follows that $K_{i_{0}}\cap G\subseteq V\cap Y=U$. So $Y$ is an $S^{\ast}$-well-filtered space.
\end{proof}

Second, we will discuss whether the image of an $S^{\ast}$-well-filtered space under a retraction mapping remains $S^{\ast}$-well-filtered.

A topological space $Y$ is a retract of a topological space $X$ if there are two continuous mappings $f: X\rightarrow Y$ and $g: Y \rightarrow X$ such that $f\circ g=id_{Y}$. It is well-known that a retract of a $T_{0}$-space is a $T_{0}$-space \cite{g2013}.
\begin{prop}\label{p12} A retract of $S^{\ast}$-well-filtered space is an $S^{\ast}$-well-filtered space.
\end{prop}
\begin{proof}
Let $X$ be an $S^{\ast}$-well-filtered space and $Y$ a retract of $X$. Then $Y$ is a $T_{0}$-space and there exist continuous mappings $f: X\rightarrow Y$ and $g: Y \rightarrow X$ such that $f\circ g=id_{Y}$. Assume that $\{K_{i}\mid i\in I\}\subseteq K(Y)$ is a filtered family, $G\in K(Y)$ and $U\in \mathcal {O}(Y){\setminus}\{\emptyset\}$ satisfying $\bigcap_{i\in I} K_{i}\cap G\subseteq U$. As $f$ and $g$ are continuous, we get that $f^{-1}(U)\in \mathcal {O}(X){\setminus}\{\emptyset\}$ and $g(K)$ is compact for any $K\in K(Y)$. Then $\{{\uparrow}g(K_{i})\mid i\in I\})\subseteq K(X)$ is a filtered family and ${\uparrow}g(G)\in K(X)$. We claim that $\bigcap_{i\in I}{\uparrow}g(K_{i})\cap{\uparrow}g(G)\subseteq f^{-1}(U)$. Indeed, for any $x\in \bigcap_{i\in I}{\uparrow}g(K_{i})\cap{\uparrow}g(G)$, one has that
\begin{align*}
f(x)\in f\left(\bigcap_{i\in I}{\uparrow}g(K_{i})\cap{\uparrow}g(G)\right)
&\subseteq f\left(\bigcap_{i\in I}{\uparrow}g(K_{i})\right)\cap f({\uparrow}g(G))
\\&\subseteq \bigcap_{i\in I}f({\uparrow}g(K_{i}))\cap f({\uparrow}g(G))
\\&\subseteq \bigcap_{i\in I}{\uparrow}(f(g(K_{i}))\cap {\uparrow}f(g(G))
\\&\subseteq \bigcap_{i\in I}{\uparrow}K_{i}\cap {\uparrow}G
\\&\subseteq \bigcap_{i\in I}K_{i}\cap G
\\&\subseteq U,
\end{align*}
that is, $x\in f^{-1}(U)$. Therefore, $\bigcap_{i\in I}{\uparrow}g(K_{i})\cap{\uparrow}g(G)\subseteq f^{-1}(U)$. Since $X$ is $S^{\ast}$-well-filtered, there exists $i_{0}\in I$ such that ${\uparrow}g(K_{i_{0}})\cap{\uparrow}g(G)\subseteq f^{-1}(U)$. For any $t\in K_{i_{0}}\cap G$, it follows that $g(t)\in g(K_{i_{0}}\cap G)\subseteq g({\uparrow}K_{i_{0}})\cap g({\uparrow}G)\subseteq {\uparrow}g(K_{i_{0}})\cap {\uparrow}g(G)\subseteq f^{-1}(U)$. Hence, $t\in g^{-1}(f^{-1}(U))=(f\circ g)^{-1}(U)=U$. That is, $K_{i_{0}}\cap G\subseteq U$ for some $i_{0}\in I$. So $Y$ is an $S^{\ast}$-well-filtered space.
\end{proof}
\begin{prop}\label{p13} Let $\{X_{i}\mid i\in I \}$ be a family
of $T_{0}$-space. If the product space $\prod _{i\in I}X_{i}$ is an $S^{\ast}$-well-filtered space, then $X_{i}$ is an $S^{\ast}$-well-filtered space for each $i\in I$.
\end{prop}
\begin{proof}
Let $X=\prod _{i\in I}X_{i}$ be an $S^{\ast}$-well-filtered space. For each $i\in I$, $X_{i}$ is a retract of $X$. It follows from Proposition \ref{p12} that $X_{i}$ is an $S^{\ast}$-well-filtered space for each $i\in I$.
\end{proof}
In \cite{x2026}, it has been proved that the product of two strongly well-filtered spaces need not be strongly well-filtered. The following example shows that the product of two $S^{\ast}$-well-filtered spaces is not an $S^{\ast}$-well-filtered space.

For the chain $2=\{0, 1\}$, the space 2 equipped with the Scott topology is called the Sierpi\'{n}ski space. In particular, the space is sober.
\begin{exmp}\label{e6}
Consider the space $\Sigma \mathbb{N}$ in Example \ref{e1} and the Sierpi\'{n}ski space $\Sigma2$. Obviously, $\Sigma \mathbb{N}$ and $\Sigma2$ are $S^{\ast}$-well-filtered. It is easy to prove that $\Sigma \mathbb{N}\times\Sigma2=\Sigma(\mathbb{N}\times2)$. We claim that $\Sigma \mathbb{N}\times\Sigma2$ is not $S^{\ast}$-well-filtered. Indeed, $\{{\uparrow}(n,0)\mid n\in \mathbb{N}\}\subseteq K(\Sigma \mathbb{N}\times\Sigma2)$ is a filtered family and $U={\uparrow}(n,1)={\uparrow}n\times\{1\}$. Then $U$ is a nonempty Scott open subset of $\Sigma \mathbb{N}\times\Sigma2$ and $\bigcap_{n\in \mathbb{N}}{\uparrow}(n,0)=\emptyset\subseteq U$. But for any $n\in \mathbb{N}$, ${\uparrow}(n,0)\nsubseteq U$. So $\Sigma \mathbb{N}\times\Sigma2$ is not a weak well-filtered space. It follows from Proposition \ref{p7} that $\Sigma \mathbb{N}\times\Sigma2$ is not an $S^{\ast}$-well-filtered space.
\end{exmp}
From Example \ref{e6} and Lemma \ref{l20}, we can immediately deduce the following result.
\begin{prop}\label{p17}
The category $\mathbf{S^{\ast}}$-$\mathbf{Top}_{w}$ is not a reflective subcategory of $\mathbf{Top}_{0}$.
\end{prop}
Finally, we will discuss the function spaces related to $S^{\ast}$-well-filtered spaces.

Given topological spaces $X$ and $Y$, let $\mathrm{TOP}(X,Y)$ be the set of all continuous functions from $X$ to $Y$.

\begin{defn}\label{d16}(\cite{g2003}) For two spaces $X$ and $Y$, the Isbell topology on the set $\mathrm{TOP}(X,Y)$ is generated by the subbasis of the form $$N(H\leftarrow V)=\left\{f\in TOP(X,Y)\mid f^{-1}(V)\in H\right\},$$
where $H$ is a Scott open subset of the complete lattice $\mathcal {O}(X)$ and $V$ is open in $Y$. Let $[X, Y]$ denote $\mathrm{TOP}(X,Y)$ endowed with the Isbell topology.
\end{defn}
\begin{lem}\label{l7}(\cite{l2021}) Let $X$ and $Y$ be two $T_{0}$ spaces. Consider the mapping
\begin{align*}
 \xi:~&Y\rightarrow [X, Y]
\\& y\mapsto \xi_{y},
\end{align*}
where $\xi_{y}(x)=y$ for all $x\in X$. Then $\xi$ is continuous.
\end{lem}
\begin{lem}\label{l8}(\cite{l2023}) Let $X$ be a nonempty $T_{0}$ space. For each $x\in X$, define a function $F_{x}:
[X, Y]\rightarrow Y$ by $F_{x}(f)=f(x)$ for any $f\in [X, Y]$. Then $F_{x}$ is continuous.
\end{lem}
It follows from Lemmas \ref{l7} and \ref{l8} that $Y$ is a retract of $[X, Y]$. By Proposition \ref{p12}, we get the following results.
\begin{prop}\label{p14} If $X$ is a nonempty $T_{0}$
space and $[X, Y]$ is an $S^{\ast}$-well-filtered space, then $Y$ is an $S^{\ast}$-well-filtered space.
\end{prop}

Conversely, one naturally asks the following question:

If $X$ is a nonempty $T_{0}$-space and $Y$ is an $S^{\ast}$-well-filtered space, then is $[X, Y]$ an $S^{\ast}$-well-filtered space?

The following example provides a negative answer to the question.
\begin{exmp}\label{e7}
Consider the space $\Sigma \mathbb{N}$ and the Sierpi\'{n}ski space $\Sigma2$ in Example \ref{e6}. For any $n\in \mathbb{N}$, define the function $f_{n}: \Sigma2\rightarrow\Sigma \mathbb{N}$ by
\[f_{n}(x)=
\begin{cases}
n,&x=1;\\
0,&x=0.
\end{cases}
\]
Then $f_{n}$ is continuous and $\{{\uparrow}f_{n}\mid n\in \mathbb{N}\}\subseteq K([\Sigma2, \Sigma \mathbb{N}])$ is a filtered family. In fact, $\{\{0,1\}\}\subseteq \mathcal {O}(\Sigma2)$ and ${\uparrow}4\in \sigma(\mathbb{N}){\setminus}\{\emptyset\}$. Let $U=N(\{\{0,1\}\}\leftarrow{\uparrow}4)=\{g\in [\Sigma2, \Sigma \mathbb{N}]\mid g^{-1}({\uparrow}4)\in \{\{0,1\}\}\}=\{g\in [\Sigma2, \Sigma \mathbb{N}]\mid g^{-1}({\uparrow}4)= \{0,1\}\}.$ Then $\bigcap_{n\in \mathbb{N}}{\uparrow}f_{n}=\emptyset\subseteq U$. But for any $n\in \mathbb{N}$, $f_{n}\notin U$ due to the fact that \[f_{n}^{-1}({\uparrow}4)=
\begin{cases}
\{1\},&n\geq4;\\
\emptyset,&n<4.
\end{cases}
\]
So ${\uparrow}f_{n}\nsubseteq U$ for any $n\in \mathbb{N}$. It implies that $[\Sigma2, \Sigma \mathbb{N}]$ is not a weak well-filtered space. So it is not an $S^{\ast}$-well-filtered space by Proposition \ref{p7}.
\end{exmp}
\section{$S^{\ast}$-well-filteredness of Smyth power spaces}\label{s6}
In this section, we focus on investigating the $S^{\ast}$-well-filteredness of Smyth power spaces.

For a topological space $X$, the set of all nonempty compact saturated subsets of $X$ is denoted by $K(X)$, a canonical topology is generated by the sets $${\Box}U=\{K'\in K(X)\mid K'\subseteq U\},$$
where $U$ ranges over the open subsets of $X$, this is the so-called upper Vietoris topology. We use $P_{S}(K(X))$ to denote the resulting topological space. The space $P_{S}(K(X))$, denoted shortly by $P_{S}(X)$, which is called the Smyth power space or upper space of $X$ (see \cite{h2013, s1993}). It is obvious that the specialization order on $P_{S}(X)$ is the reverse inclusion order, that is, $K_{1}, K_{2}\in K(X)$, $K_{1}\leq K_{2}$ if and only if $K_{2}\subseteq K_{1}$.
\begin{prop}\label{p15} Let $X$ be
a $T_{0}$-space. If $P_{S}(X)$ is a $d^{\ast}$-space, then $X$ is an $S^{\ast}$-well-filtered space.
\end{prop}
\begin{proof}
Assume that $\{K_{i}\mid i\in I\}\subseteq K(X)$ is a filtered family, $G\in K(X)$ and $U\in \mathcal {O}(X){\setminus}\{\emptyset\}$ satisfying $\bigcap_{i\in I} K_{i}\cap G\subseteq U$. Then $\{K_{i}\mid i\in I\}\in \mathcal {D}(P_{S}(X))$ and $\bigcap_{i\in I}{\uparrow_{P_{S}(X)}} K_{i}\cap {\uparrow_{P_{S}(X)}}G=\{K\in P_{S}(X)\mid K\subseteq\bigcap_{i\in I} K_{i}\cap G\}\subseteq {\Box}U$. Because $P_{S}(X)$ is a $d^{\ast}$-space, there exists $i_{0}\in I$ such that ${\uparrow_{P_{S}(X)}}K_{i_{0}}\cap {\uparrow_{P_{S}(X)}}G\subseteq {\Box}U$. For any $x\in K_{i_{0}}\cap G$, one has that ${\uparrow}x\in {\uparrow_{P_{S}(X)}}K_{i_{0}}\cap {\uparrow_{P_{S}(X)}}G\subseteq {\Box}U$. Therefore, ${\uparrow}x\subseteq U$, that is, $x\in U$. So $X$ is an $S^{\ast}$-well-filtered space.
\end{proof}
The following corollary is immediately obtained by Propositons \ref{p8} and \ref{p15}
\begin{cor}\label{c4}(\cite{c2023}) Let $X$ be
a $T_{0}$-space. If $P_{S}(X)$ is a $d^{\ast}$-space, then $X$ is a $d^{\ast}$-space.
\end{cor}
\begin{lem}\label{l9}(\cite{l2017}) For a weak well-filtered space $X$ and a filtered family $\mathcal {K}\subseteq K(X)$, $\bigcap\mathcal {K}\in K(X)$.
\end{lem}
\begin{prop}\label{p16} Let $X$ be
a $T_{0}$-space. If $X$ is a coherent weak well-filtered space, then $P_{S}(X)$ is a $d^{\ast}$-space.
\end{prop}
\begin{proof}
Let $X$ be a coherent weak well-filtered space. Assume that $\{K_{i}\mid i\in I\}\in \mathcal {D}(P_{S}(X))$, $G\in P_{S}(X)$ and $\mathcal {U}\in \mathcal {O}(P_{S}(X)){\setminus}\{\emptyset\}$ satisfying $\bigcap_{i\in I}{\uparrow_{P_{S}(X)}} K_{i}\cap {\uparrow_{P_{S}(X)}}G\subseteq \mathcal {U}$. Then $\{K_{i}\mid i\in I\}\subseteq K(X)$ is a filtered family and $\mathcal {U}=\bigcup_{j\in J}{\Box}U_{j}$ for some $\{U_{j}\mid j\in J\}\subseteq\mathcal {O}(X){\setminus}\{\emptyset\}$. If there exists $i_{0}\in I$ such that $K_{i_{0}}\cap G=\emptyset$, then ${\uparrow_{P_{S}(X)}}K_{i_{0}}\cap {\uparrow_{P_{S}(X)}}G\subseteq \mathcal {U}$. Now suppose that $K_{i}\cap G\neq\emptyset$ for all $i\in I$. Then for all $i\in I$, $K_{i}\cap G$ is compact since $X$ is coherent. Thus, $\{K_{i}\cap G\mid i\in I\}\subseteq K(X)$ is a filtered family and $\bigcap_{i\in I }(K_{i}\cap G)\in K(X)$ from Lemma \ref{l9}.
So ${\uparrow_{P_{S}(X)}}\bigcap_{i\in I }(K_{i}\cap G)=\bigcap_{P_{S}(X)}{\uparrow_{P_{S}(X)}}K_{i}\cap{\uparrow_{P_{S}(X)}}G\subseteq \mathcal {U}=\bigcup_{j\in J}{\Box}U_{j}$. It implies that there exists $j_{0}\in J$ such that $\bigcap_{i\in I }(K_{i}\cap G)\in {\Box}U_{j_{0}}$. Hence, $\bigcap_{i\in I }(K_{i}\cap G)\subseteq U_{j_{0}}$. By the weak well-filteredness of $X$, there exists $i_{1}\in I$ such that $K_{i_{1}}\cap G\subseteq U_{j_{0}}$. Therefore, ${\uparrow_{P_{S}(X)}}K_{i_{1}}\cap{\uparrow_{P_{S}(X)}}G={\uparrow_{P_{S}(X)}}K_{i_{1}}\cap G\subseteq {\Box}U_{j_{0}}\subseteq\mathcal {U}$. Thus, $P_{S}(X)$ is a $d^{\ast}$-space.
\end{proof}
The following two corollaries are trivial by Propositons \ref{p7} and \ref{p16}.
\begin{cor}\label{c5} Let $X$ be
a $T_{0}$-space. If $X$ is a coherent $S^{\ast}$-well-filtered space, then $P_{S}(X)$ is a $d^{\ast}$-space.
\end{cor}
\begin{cor}\label{c6} Let $X$ be a KC space. If $X$ is $S^{\ast}$-well-filtered, then $P_{S}(X)$ is a $d^{\ast}$-space.
\end{cor}
From Propositions \ref{p7}, \ref{p8}, \ref{p15} and \ref{p16}, we deduce the following statement.
\begin{thm}\label{t2} Let $X$ be
a $T_{0}$-space. Consider the following conditions:
\begin{enumerate}[(1)]
  \item\label{t2i1} $P_{S}(X)$ is an $S^{\ast}$-well-filtered space;
  \item\label{t2i2} $P_{S}(X)$ is a $d^{\ast}$-space;
  \item\label{t2i3} $X$ is an $S^{\ast}$-well-filtered space;
  \item\label{t2i4} $X$ is a weak well-filtered space.
\end{enumerate}
Then $(\ref{t2i1})\Rightarrow(\ref{t2i2})\Rightarrow(\ref{t2i3})\Rightarrow(\ref{t2i4})$. Moreover, $(\ref{t2i2})\Leftrightarrow(\ref{t2i3})\Leftrightarrow(\ref{t2i4})$ if $X$ is coherent.
\end{thm}
\begin{cor}\label{c7} Let $X$ be a $T_{0}$-space. Consider the following conditions:
\begin{enumerate}[(1)]
  \item\label{c7i1} $P_{S}(X)$ is an $S^{\ast}$-well-filtered space;
  \item\label{c7i2} $P_{S}(X)$ is a $d^{\ast}$-space;
  \item\label{c7i3} $X$ is an $S^{\ast}$-well-filtered space;
  \item\label{c7i4} $X$ is a weak well-filtered space.
\end{enumerate}
Then $(\ref{c7i1})\Rightarrow(\ref{c7i2})\Rightarrow(\ref{c7i3})\Rightarrow(\ref{c7i4})$. Moreover, $(\ref{t2i2})\Leftrightarrow(\ref{t2i3})\Leftrightarrow(\ref{t2i4})$ if $X$ is a KC space.
\end{cor}
For a $T_{0}$-space $X$ and $U\in \mathcal {O}(X)$, let $\diamond U= \{A\in \Gamma(X)\mid A\cap U\neq\emptyset\}$. The lower Vietoris topology on $\Gamma(X){\setminus}\{\emptyset\}$ is the topology that has $\{\diamond V \mid V\in \mathcal {O}(X)\}$ as a subbasis and the resulting space is denoted by $P_{H}(X)$, called the Hoare power space or lower space of $X$ \cite{s1993}.
\begin{lem}\label{l19}(\cite{x2025}) For a $T_{0}$-space $X$, Hoare power space $P_{H}(X)$ is a strongly well-filtered space.
\end{lem}
\begin{cor}\label{c8}
For a $T_{0}$-space $X$, Hoare power space $P_{H}(X)$ is an $S^{\ast}$-well-filtered space.
\end{cor}
%\section*{Acknowledgements}
%The authors would like to express their sincere thanks to the Editors and anonymous reviewers for their most valuable comments and suggestions in improving this paper greatly.

\bibliographystyle{elsarticle-num}

\end{document}